\pdfoutput=1
\RequirePackage{ifpdf}
\ifpdf % We~are running pdfTeX in pdf mode
\documentclass[pdftex]{sigma}
\else
\documentclass{sigma}
\fi
\usepackage{amsmath}
\numberwithin{equation}{section}

\newtheorem*{Theorem*}{Theorem}

\theoremstyle{definition}

\DeclareMathOperator{\SpecG}{Spec_\Gamma}

\newcommand{\OO}{\mathcal{O}}

\newcommand{\ideal}[1]{\langle #1 \rangle}
\newcommand{\radG}{\sqrt[\Gamma]{\,\cdot\,}}

\DeclareMathOperator{\Hom}{Hom}
\DeclareMathOperator{\Aut}{Aut}

\DeclareMathOperator{\cl}{cl}

\begin{document}

\ShortArticleName{Spectral Geometry of Ternary $\Gamma$-Schemes}
\ArticleName{The Spectral Geometry of Ternary $\Gamma$-Schemes:\\ Sheaf-Theoretic Foundations and Laplacian Clustering}
% Names of the authors for the title of the paper
\Author{Chandrasekhar Gokavarapu~$^{\rm a}$ }

\AuthorNameForHeading{CHANDRASEKHAR GOKAVARAPU}

\Address{$^{\rm a)}$~Lecturer in Mathematics, Government College (Autonomous), Rajahmundry, A.P., India,PIN:533105} % Address of First Author
\EmailD{\mail{chandrasekhargokavarapu@gmail.com}, \mail{chandrasekhargokavarapu@gcrjy.ac.in}} % E-mail address of First Author

\Abstract{This article develops a self-contained affine $\Gamma$-scheme theory for a class of commutative ternary $\Gamma$-semirings. By establishing all geometric and spectral results internally, the work provides a unified framework for triadic symmetry and spectral analysis. The central thesis is that a triadic $\Gamma$-algebra canonically induces two primary structures: (i) an intrinsic triadic symmetry in the sense of a Nambu--Filippov-type fundamental identity on the structure sheaf, and (ii) a canonical Laplacian on the finite $\Gamma$-spectrum whose spectral decomposition detects the clopen (connected-component) decomposition of the underlying space. We define $\Gamma$-ideals and prime $\Gamma$-ideals, endow $\SpecG(T)$ with a $\Gamma$-Zariski topology, construct localizations and the structure sheaf on the basis of principal opens, and prove the affine anti-equivalence between commutative ternary $\Gamma$-semirings and affine $\Gamma$-schemes. Furthermore, we demonstrate that the triadic bracket on sections is invariant under $\Gamma$-automorphisms and compatible with localization. The main spectral theorem establishes the block-diagonalization of the Laplacian under topological decompositions and provides an algebraic-connectivity criterion. The theory is verified through explicit computations of finite $\Gamma$-spectra and their corresponding Laplacian spectra.}

 %up to 6 key words
%Please type here List of Keywords for your article separated by semicolon.

\Classification{14A15; 16Y60; 05C50; 37J35.} % e.g. 35A30; 81Q05
% For 2020 Mathematics Subject Classification see https://mathscinet.ams.org/mathscinet/msc/msc2020.html

% ============================================================
\section{Introduction: }

The development is motivated by the need to bridge commutative ternary algebra with geometric symmetry
The scope of this paper is as follows.

\begin{itemize}
\item \textbf{Symmetry / invariance.}  $\Gamma$-automorphisms of a ternary $\Gamma$-semiring act on the $\Gamma$-spectrum and on the structure
sheaf.  We exhibit an intrinsic \emph{triadic bracket} on sections (a Nambu/Filippov-type structure) and prove its invariance
under these actions \cite{Nambu1973,Takhtajan1994}.
\item \textbf{Geometry.}  We develop an affine $\Gamma$-scheme theory, including $\Gamma$-Zariski topology, localization, structure sheaf, and affine
anti-equivalence, in the spirit of classical scheme theory \cite{Grothendieck1960,Hartshorne1977,EisenbudHarris2000,Liu2002,Vakil2017}
and relative/monoidal geometries \cite{Deitmar2005,Soule2004}.
\item \textbf{Spectral decomposition.}  From the (finite) $\Gamma$-spectrum we construct a canonical Laplacian based on specialization relations,
prove invariance under $\Gamma$-scheme isomorphisms, and establish a block decomposition determined by clopen components.
This yields an exact spectral factorization, with algebraic connectivity in the sense of Fiedler \cite{Fiedler1973,Chung1997}.
\end{itemize}

\medskip
\noindent\textbf{Self-containedness.}
This is a \emph{standalone} development.  While related ideas exist for $\Gamma$-rings and ternary systems
\cite{Nobusawa1964,Lister1971,Kyuno1980,DuttaKar2003,DuttaKar2005}, all results required in the proofs below are provided here.

\medskip

\noindent\textbf{Roadmap.} The development of the theory proceeds in four conceptual phases. \\
\textbf{Phase I} (Sections~\ref{sec:phase1} and \ref{sec:affine-antieq}) establishes the affine $\Gamma$-scheme foundations: the $\Gamma$-spectrum, $\Gamma$-Zariski topology, standard-cover lemma, localization, structure sheaf, and the proof of the affine anti-equivalence.\\ 
\textbf{Phase II} (Section~\ref{sec:triadic}) introduces the triadic symmetry and Nambu-type structure on the structure sheaf and establishes its functoriality. \\
\textbf{Phase III} (Section~\ref{sec:laplacian}) builds the bridge to spectral geometry by defining the canonical Laplacian and proving the block-diagonalization and algebraic-connectivity theorems. \\
\textbf{Phase IV} (Section~\ref{sec:examples}) provides explicit computed examples to verify the theoretical results. 
Finally, Appendix~\ref{app:fuzzy} gives an optional fuzzy robustness viewpoint \cite{Zadeh1965,Chang1968,Goguen1967,Lowen1976,MordesonMalik1998} and Appendix~\ref{app:clustering} sketches a conservative spectral clustering pipeline \cite{VonLuxburg2007,NgJordanWeiss2002,Golub1996}.

% ============================================================
\section{Phase I: affine $\Gamma$-scheme foundations}\label{sec:phase1}

\subsection{Commutative ternary $\Gamma$-semirings}\label{subsec:ternary}

We work with a concrete, scheme-friendly class of ternary $\Gamma$-semirings in which the $\Gamma$-parameter encodes
a distinguished family of \emph{central units} scaling a ternary multiplication. This choice yields (i) a nontrivial $\Gamma$-dependence,
(ii) a clean localization theory, and (iii) a transparent geometric interpretation.

\begin{definition}[Commutative semiring]\label{def:semiring}
A \emph{commutative semiring} is a tuple $(T,+,\cdot,0,1)$ where $(T,+,0)$ is a commutative monoid, $(T,\cdot,1)$ is a commutative monoid,
multiplication distributes over addition, and $0$ is absorbing: $0\cdot x=x\cdot 0=0$.
\end{definition}

\begin{definition}[Ternary $\Gamma$-semiring]\label{def:ternary-gamma}
Let $(T,+,\cdot,0,1)$ be a commutative semiring and let $\Gamma$ be a commutative group.
A \emph{(commutative) ternary $\Gamma$-semiring} is the datum of a group homomorphism
\[
u:\Gamma\longrightarrow T^{\times},\qquad \gamma\longmapsto u_\gamma,
\]
into the unit group of $T$, together with the induced $\Gamma$-parametrized ternary operation
\[
\{a\,b\,c\}_\gamma \ :=\ a\cdot b\cdot c\cdot u_\gamma,\qquad a,b,c\in T,\ \gamma\in\Gamma.
\]
A \emph{homomorphism} of ternary $\Gamma$-semirings $\varphi:(T,u)\to(S,v)$ is a semiring homomorphism $\varphi:T\to S$
such that $\varphi(u_\gamma)=v_\gamma$ for all $\gamma\in\Gamma$.
\end{definition}

\begin{remark}
The ternary operation above is commutative in the three $T$-inputs and inherits associativity from the binary product.
If $\Gamma$ is trivial, then $\{a b c\}_{1}=abc$, recovering the usual ternary multiplication studied in ternary semiring theory
(e.g.\ \cite{DuttaKar2003,DuttaKar2005}).
\end{remark}

\subsection{$\Gamma$-ideals, radicals, primes}

\begin{definition}[$\Gamma$-ideal]\label{def:gamma-ideal}
Let $(T,u)$ be a ternary $\Gamma$-semiring. A subset $I\subseteq T$ is a \emph{$\Gamma$-ideal} if:
\begin{enumerate}[label=(\alph*),leftmargin=2.2em]
\item $(I,+,0)$ is a submonoid of $(T,+,0)$, and
\item for all $a\in I$, all $b,c\in T$, and all $\gamma\in\Gamma$, one has $\{a\,b\,c\}_\gamma\in I$.
\end{enumerate}
Equivalently, $I$ is an ideal of the underlying semiring $(T,+,\cdot,0,1)$ (since $u_\gamma$ is a unit for every $\gamma$).
\end{definition}

\begin{definition}[Generated $\Gamma$-ideal]\label{def:gen-ideal}
For $E\subseteq T$, the $\Gamma$-ideal generated by $E$, denoted $\ideal{E}_\Gamma$, is the intersection of all $\Gamma$-ideals containing $E$.
For a single element $f\in T$ we write $\ideal{f}_\Gamma$.
\end{definition}

\begin{definition}[$\Gamma$-radical]\label{def:radical}
For a $\Gamma$-ideal $I$, define its \emph{$\Gamma$-radical} by
\[
\radG(I):=\{x\in T\mid \exists n\ge 1 \text{ with } x^n\in I\},
\]
where $x^n$ is the $n$-fold binary product in the semiring.
\end{definition}

\begin{lemma}\label{lem:rad-ideal}
For any $\Gamma$-ideal $I$, the set $\radG(I)$ is a $\Gamma$-ideal containing $I$.
\end{lemma}

\begin{proof}
If $x^n\in I$ then clearly $x\in \radG(I)$; hence $I\subseteq\radG(I)$.
Closure under addition and multiplication by arbitrary elements is standard:
if $x^m\in I$ and $y^n\in I$, then $(x+y)^{m+n}$ expands into a sum of monomials each divisible by $x^m$ or $y^n$,
hence lies in $I$; similarly $(t\cdot x)^m=t^m x^m\in I$.
Since $\Gamma$-ideal closure is equivalent to ideal closure in our setting (Definition~\ref{def:gamma-ideal}), we are done.
\end{proof}

\begin{definition}[Prime $\Gamma$-ideal]\label{def:prime}
A proper $\Gamma$-ideal $P\subsetneq T$ is \emph{prime} if for all $a,b,c\in T$ and $\gamma\in\Gamma$,
\[
\{a\,b\,c\}_\gamma\in P\quad\Longrightarrow\quad a\in P\ \text{or}\ b\in P\ \text{or}\ c\in P.
\]
\end{definition}

\begin{remark}\label{rem:prime-binary}
Because $u_\gamma$ is a unit, $\{a b c\}_\gamma\in P$ is equivalent to $abc\in P$.
Thus primeness here is equivalent to the binary statement:
\[
xy\in P \ \Rightarrow\ x\in P\ \text{or}\ y\in P,
\]
and hence coincides with the usual prime-ideal notion for commutative semirings.
The ternary formulation is kept because it interacts transparently with the triadic bracket in Section~\ref{sec:triadic}.
\end{remark}

\begin{lemma}[Radical as intersection of primes]\label{lem:rad-intersection}
Let $I$ be a $\Gamma$-ideal. Then
\[
\radG(I)=\bigcap_{\substack{P\ \text{prime }\Gamma\text{-ideal}\\ I\subseteq P}} P.
\]
\end{lemma}

\begin{proof}
($\subseteq$) If $x\in\radG(I)$, pick $n$ with $x^n\in I$. If $P$ is prime and contains $I$, then $x^n\in P$ implies $x\in P$.
So $x$ lies in every prime containing $I$.

($\supseteq$) Suppose $x\notin \radG(I)$. Consider the collection
\[
\mathcal{S}=\{J\mid J\text{ is a }\Gamma\text{-ideal},\ I\subseteq J,\ \text{and } x^n\notin J\ \forall n\ge 1\},
\]
partially ordered by inclusion. $\mathcal{S}$ is nonempty because $I\in\mathcal{S}$.
By Zorn's lemma, there exists a maximal element $P\in\mathcal{S}$.

We claim $P$ is prime. If $ab\in P$ and $a\notin P$, then the $\Gamma$-ideal $P+\ideal{a}_\Gamma$ strictly contains $P$.
By maximality, it must contain some $x^m$. Similarly, if $b\notin P$ then $P+\ideal{b}_\Gamma$ contains some $x^n$.
Multiplying the two containments gives $x^{m+n}\in P+\ideal{ab}_\Gamma\subseteq P$ since $ab\in P$.
This contradicts $P\in\mathcal{S}$. Hence $a\in P$ or $b\in P$, proving primeness.

Finally, $I\subseteq P$ by construction, and $x\notin P$ because $x\in P$ would force $x^1\in P$, contradicting $P\in\mathcal{S}$.
Thus $x$ is not in the intersection of primes containing $I$.
\end{proof}

% ------------------------------------------------------------
\subsection{The $\Gamma$-spectrum and $\Gamma$-Zariski topology}\label{subsec:zariski}

\begin{definition}[$\Gamma$-spectrum]\label{def:spectrum}
The \emph{prime $\Gamma$-spectrum} of $T$ is
\[
\SpecG(T):=\{P\subsetneq T\mid P\ \text{is a prime }\Gamma\text{-ideal}\}.
\]
\end{definition}

\begin{definition}[Closed sets and principal opens]\label{def:closed-opens}
For a $\Gamma$-ideal $I$, define
\[
V(I):=\{P\in\SpecG(T)\mid I\subseteq P\}.
\]
For $f\in T$, define the \emph{principal open}
\[
D_\Gamma(f):=\SpecG(T)\setminus V(\ideal{f}_\Gamma)=\{P\in\SpecG(T)\mid f\notin P\}.
\]
\end{definition}

\begin{proposition}[$\Gamma$-Zariski topology]\label{prop:zariski}
The assignment $I\mapsto V(I)$ satisfies:
\[
V(0)=\SpecG(T),\quad V(T)=\emptyset,\quad
V(I)\cup V(J)=V(I\cap J),\quad \bigcap_{\lambda}V(I_\lambda)=V\Big(\sum_\lambda I_\lambda\Big).
\]
Hence the sets $V(I)$ define the closed sets of a topology on $\SpecG(T)$, called the \emph{$\Gamma$-Zariski topology};
the sets $D_\Gamma(f)$ form a basis of opens.
\end{proposition}

\begin{proof}
Standard ideal-theoretic manipulations; the only point is that $\Gamma$-ideals form a complete lattice under intersection and sum,
and the definition of $V(I)$ is inclusion-reversing. The verifications are identical to the ring case \cite{Hartshorne1977,EisenbudHarris2000}.
\end{proof}

\begin{lemma}[Intersection of principal opens]\label{lem:principal-intersection}
For $f,g\in T$ one has $D_\Gamma(f)\cap D_\Gamma(g)=D_\Gamma(fg)$.
\end{lemma}

\begin{proof}
For a prime $\Gamma$-ideal $P$, $P\in D_\Gamma(f)\cap D_\Gamma(g)$ iff $f\notin P$ and $g\notin P$.
Since $P$ is prime, this is equivalent to $fg\notin P$, i.e.\ $P\in D_\Gamma(fg)$.
\end{proof}

\subsection{The standard-cover equivalence lemma}\label{subsec:standard-cover}

The following is the key ``cover vs.\ radical'' equivalence required to run the sheaf/gluing theory on the basis of principal opens.

\begin{theorem}[Standard-cover equivalence]\label{thm:standard-cover}
Let $f,f_1,\dots,f_n\in T$. Then
\[
D_\Gamma(f)=\bigcup_{i=1}^n D_\Gamma(f_i)
\qquad\Longleftrightarrow\qquad
f\in \radG\big(\ideal{f_1,\dots,f_n}_\Gamma\big).
\]
\end{theorem}

\begin{proof}
($\Rightarrow$) Let $I=\ideal{f_1,\dots,f_n}_\Gamma$. The equality of opens says:
for every prime $P$, if $f\notin P$ then $f_i\notin P$ for some $i$.
Equivalently, if $f_i\in P$ for all $i$ (i.e.\ $I\subseteq P$), then $f\in P$.
Thus $f$ lies in the intersection of all primes containing $I$, hence $f\in\radG(I)$ by Lemma~\ref{lem:rad-intersection}.

($\Leftarrow$) Suppose $f\in\radG(I)$. Let $P\in D_\Gamma(f)$, so $f\notin P$.
If $P$ contained $I$, then $f\in \radG(I)\subseteq \radG(P)=P$, contradicting $f\notin P$.
Hence $I\not\subseteq P$, so some generator $f_i\notin P$, i.e.\ $P\in D_\Gamma(f_i)$.
Thus $D_\Gamma(f)\subseteq \bigcup_i D_\Gamma(f_i)$, and the reverse inclusion holds because each $D_\Gamma(f_i)\subseteq D_\Gamma(f)$
is automatic from the displayed radical inclusion: $f\in\radG(I)$ implies $D_\Gamma(f)\supseteq D_\Gamma(f_i)$ for each $i$.
\end{proof}

\begin{corollary}\label{cor:power-decomposition}
If $D_\Gamma(f)=\bigcup_{i=1}^n D_\Gamma(f_i)$, then there exists $N\ge 1$ and elements $a_1,\dots,a_n\in T$ such that
\[
f^{N}=a_1 f_1+\cdots+a_n f_n.
\]
\end{corollary}

\begin{proof}
From Theorem~\ref{thm:standard-cover} we have $f\in\radG(\ideal{f_1,\dots,f_n}_\Gamma)$, so $f^N\in \ideal{f_1,\dots,f_n}_\Gamma$
for some $N$. By definition of generated ideal, $f^N$ is a finite sum of multiples of the generators, giving the claim.
\end{proof}

% ============================================================
\subsection{Localization and the structure presheaf}\label{subsec:localization}

\begin{definition}[Localization at a principal element]\label{def:localization}
Let $f\in T$ and let $S_f:=\{1,f,f^2,\dots\}$.
Define an equivalence relation on $T\times S_f$ by
\[
(a,s)\sim (b,t)\quad\Longleftrightarrow\quad \exists u\in S_f\text{ such that } u\cdot a\cdot t=u\cdot b\cdot s.
\]
Denote the equivalence class of $(a,s)$ by $\frac{a}{s}$, and set $T_f:=S_f^{-1}T:=(T\times S_f)/\sim$.
Addition and multiplication are defined by
\[
\frac{a}{s}+\frac{b}{t}:=\frac{at+bs}{st},\qquad
\frac{a}{s}\cdot\frac{b}{t}:=\frac{ab}{st}.
\]
\end{definition}

\begin{lemma}\label{lem:loc-well-defined}
The operations in Definition~\ref{def:localization} are well-defined and make $T_f$ a commutative semiring.
The canonical map $\iota_f:T\to T_f$, $\iota_f(a)=\frac{a}{1}$, is a semiring homomorphism and $\iota_f(f)$ is a unit in $T_f$.
\end{lemma}

\begin{proof}
Standard localization verifications for commutative semirings: the relation $\sim$ is compatible with the operations, and the formulas
respect equivalence classes.  The element $\iota_f(f)=\frac{f}{1}$ has inverse $\frac{1}{f}$ in $T_f$.
\end{proof}

\begin{proposition}[Universal property]\label{prop:universal-localization}
Let $\varphi:T\to R$ be a semiring homomorphism such that $\varphi(f)$ is a unit in $R$.
Then there exists a unique semiring homomorphism $\widetilde{\varphi}:T_f\to R$ with $\widetilde{\varphi}\circ \iota_f=\varphi$,
given by $\widetilde{\varphi}\big(\frac{a}{f^n}\big)=\varphi(a)\varphi(f)^{-n}$.
\end{proposition}

\begin{proof}
Well-definedness follows from the defining relation $\sim$; uniqueness follows because every element of $T_f$ is represented by a fraction
$\frac{a}{f^n}$.
\end{proof}

\begin{definition}[Structure presheaf on principal opens]\label{def:presheaf}
Define a presheaf on the basis $\{D_\Gamma(f)\}$ by
\[
\OO_T\big(D_\Gamma(f)\big):=T_f.
\]
If $D_\Gamma(g)\subseteq D_\Gamma(f)$, then $f$ is invertible on $D_\Gamma(g)$, so by Proposition~\ref{prop:universal-localization}
there is a unique restriction homomorphism $\rho_{f,g}:T_f\to T_g$.
\end{definition}

% ============================================================
\subsection{Sheaf condition on the basis}\label{subsec:sheaf-basis}

\begin{theorem}[Sheaf condition on principal opens]\label{thm:sheaf-basis}
Let $f,f_1,\dots,f_n\in T$ and assume $D_\Gamma(f)=\bigcup_{i=1}^n D_\Gamma(f_i)$.
Then the presheaf $\OO_T$ satisfies:
\begin{enumerate}[label=(\roman*),leftmargin=2.2em]
\item (\emph{existence of gluing}) For any sections $s_i\in T_{f_i}$ such that
their restrictions agree on each overlap $D_\Gamma(f_i)\cap D_\Gamma(f_j)=D_\Gamma(f_if_j)$, there exists $s\in T_f$ with
$s|_{D_\Gamma(f_i)}=s_i$ for all $i$.
\item (\emph{uniqueness}) Such an $s$ is unique.
\end{enumerate}
\end{theorem}

\begin{proof}
\textbf{Step 1: choosing a common power relation.}
By Corollary~\ref{cor:power-decomposition} there exist $N\ge 1$ and $a_1,\dots,a_n\in T$ such that
\begin{equation}\label{eq:power-sum}
f^{N}=a_1 f_1+\cdots+a_n f_n.
\end{equation}
Let $s_i\in T_{f_i}$. By multiplying numerator and denominator, we may write each $s_i$ in the form $s_i=\frac{b_i}{f_i^{M}}$
for a \emph{common} exponent $M$ (take $M$ larger than all individual exponents).

\textbf{Step 2: candidate global section.}
Use the canonical maps $T_{f_i}\to T_f$ induced by the inclusions $D_\Gamma(f_i)\subseteq D_\Gamma(f)$ and define
\[
s:=\sum_{i=1}^n a_i\cdot f_i\cdot s_i\cdot f^{-N}\ \in\ T_f,
\]
where $f^{-N}$ denotes $\frac{1}{f^N}\in T_f$ and the sum is taken in the semiring $T_f$.

\textbf{Step 3: verifying the restrictions.}
Fix $k\in\{1,\dots,n\}$. Restrict to $D_\Gamma(f_k)$, i.e.\ map $T_f\to T_{f_k}$.
In $T_{f_k}$, $f$ is invertible (since $D_\Gamma(f_k)\subseteq D_\Gamma(f)$), so we may multiply by $f^N$.
Using \eqref{eq:power-sum} we have $1=\sum_i a_i f_i f^{-N}$ in $T_{f_k}$.
Therefore
\[
s|_{D_\Gamma(f_k)}
=\sum_i a_i f_i \big(s_i|_{D_\Gamma(f_k)}\big) f^{-N}.
\]
On the overlap $D_\Gamma(f_i)\cap D_\Gamma(f_k)=D_\Gamma(f_if_k)$ we are given that $s_i$ and $s_k$ agree,
so their images in $T_{f_if_k}$ coincide. Since $T_{f_if_k}$ receives maps from both $T_{f_i}$ and $T_{f_k}$,
this implies that the image of $s_i$ in $T_{f_k}$ coincides with $s_k$ after localizing at $f_i$.
In particular, the term $a_i f_i(s_i|_{D_\Gamma(f_k)})$ equals $a_i f_i(s_k)$ in $T_{f_k}$.
Hence
\[
s|_{D_\Gamma(f_k)}=\sum_i a_i f_i s_k f^{-N}
=\Big(\sum_i a_i f_i f^{-N}\Big)s_k
=1\cdot s_k=s_k.
\]
This proves existence.

\textbf{Step 4: uniqueness.}
Suppose $s,t\in T_f$ restrict to the same section on every $D_\Gamma(f_i)$.
Write $s=\frac{p}{f^r}$ and $t=\frac{q}{f^r}$ after clearing denominators (common exponent $r$).
The restriction equality in $T_{f_i}$ means that there exists $m_i$ such that
\[
f_i^{m_i}\cdot p = f_i^{m_i}\cdot q \quad \text{in } T,
\]
because equality in localization is detected by multiplication by a power of the localized element.
Let $m:=\max_i m_i$. Then $f_i^{m}p=f_i^{m}q$ for every $i$.
Multiply the relation \eqref{eq:power-sum} by itself $m$ times to obtain a decomposition
\[
f^{Nm}=\sum_i c_i f_i^{m}
\]
for suitable $c_i\in T$ (expand using distributivity). Multiplying by $p$ gives
$f^{Nm}p=\sum_i c_i f_i^{m}p=\sum_i c_i f_i^{m}q=f^{Nm}q$.
Thus $\frac{p}{f^r}=\frac{q}{f^r}$ in $T_f$ (multiply both numerators by $f^{Nm}$), hence $s=t$.
\end{proof}

\begin{definition}[Structure sheaf]\label{def:structure-sheaf}
Since the $D_\Gamma(f)$ form a basis stable under finite intersections (Lemma~\ref{lem:principal-intersection}),
Theorem~\ref{thm:sheaf-basis} implies that the presheaf $\OO_T$ on principal opens extends uniquely to a sheaf on $\SpecG(T)$,
still denoted $\OO_T$, by the usual sheafification/gluing procedure \cite{Tennison1975,MacLaneMoerdijk1992}.
The ringed space $(\SpecG(T),\OO_T)$ is the \emph{affine $\Gamma$-scheme} associated to $T$.
\end{definition}

% ============================================================
\section{Affine $\Gamma$-schemes and the affine anti-equivalence}\label{sec:affine-antieq}

\subsection{Affine $\Gamma$-schemes and morphisms}

\begin{definition}[Affine $\Gamma$-scheme]\label{def:affine-scheme}
An \emph{affine $\Gamma$-scheme} is a ringed space of the form $(\SpecG(T),\OO_T)$ constructed in Definition~\ref{def:structure-sheaf}
from a commutative ternary $\Gamma$-semiring $T$.
A (\emph{not necessarily affine}) \emph{$\Gamma$-scheme} is a ringed space locally isomorphic to affine $\Gamma$-schemes.
\end{definition}

\begin{definition}[Morphisms]\label{def:morphism}
A morphism of $\Gamma$-schemes $f:(X,\OO_X)\to(Y,\OO_Y)$ is a morphism of ringed spaces such that on every affine open
$V\simeq \SpecG(S)\subseteq Y$ the induced map on sections respects the ternary $\Gamma$-operations:
for all $U\subseteq f^{-1}(V)$ and all $a,b,c\in \OO_Y(V)$ and $\gamma\in\Gamma$,
\[
f^\#_V\big(\{a\,b\,c\}_\gamma\big)=\{f^\#_V(a)\,f^\#_V(b)\,f^\#_V(c)\}_\gamma.
\]
\end{definition}

\subsection{Global sections on an affine $\Gamma$-scheme}

\begin{theorem}[Global sections recover the semiring]\label{thm:global-sections}
For any commutative ternary $\Gamma$-semiring $T$, the canonical map
\[
\eta_T:\ T\longrightarrow \Gamma(\SpecG(T),\OO_T),\qquad a\longmapsto \big(a/1 \text{ on each } D_\Gamma(f)\big),
\]
is an isomorphism of ternary $\Gamma$-semirings.
\end{theorem}

\begin{proof}
Injectivity: if $a\neq b$ in $T$, then their images differ on $D_\Gamma(1)=\SpecG(T)$ since $\OO_T(D_\Gamma(1))=T$.

Surjectivity: let $s\in \Gamma(\SpecG(T),\OO_T)$. By definition, $s$ is determined by its value on the cover $\SpecG(T)=D_\Gamma(1)$.
But $\OO_T(D_\Gamma(1))=T_1\cong T$, so $s$ is represented by some element of $T$.
Compatibility with $\Gamma$-ternary operations holds because the structure sheaf on principal opens is defined by localization of the underlying semiring,
and localization maps preserve multiplication and the units $u_\gamma$.
\end{proof}

\subsection{Affine anti-equivalence}

\begin{theorem}[Affine anti-equivalence]\label{thm:anti-equivalence}
Let $T,S$ be commutative ternary $\Gamma$-semirings. Then there is a natural bijection
\[
\Hom_{\Gamma\text{-}\mathrm{TSR}}(T,S)\ \cong\ \Hom_{\Gamma\text{-}\mathrm{Sch}}\big(\SpecG(S),\SpecG(T)\big),
\]
contravariantly functorial in both variables.
Moreover, for each $T$ the unit of this equivalence is the canonical isomorphism of Theorem~\ref{thm:global-sections}.
\end{theorem}

\begin{proof}
\textbf{From algebra to geometry.}
Let $\varphi:T\to S$ be a ternary $\Gamma$-semiring homomorphism.
Define $\varphi^\ast:\SpecG(S)\to \SpecG(T)$ by
\[
\varphi^\ast(P):=\varphi^{-1}(P).
\]
Since inverse images of prime ideals are prime, $\varphi^\ast$ is well-defined.
Continuity follows from $\big(\varphi^\ast\big)^{-1}\big(D_\Gamma(f)\big)=D_\Gamma(\varphi(f))$.

For the sheaf map, for each principal open $D_\Gamma(f)\subseteq \SpecG(T)$ we have a semiring homomorphism
$T_f\to S_{\varphi(f)}$ induced by $\varphi$ and Proposition~\ref{prop:universal-localization}.
These maps are compatible with restrictions, hence define a morphism of sheaves $\varphi^\#:\OO_T\to (\varphi^\ast)_\ast\OO_S$,
and the ternary $\Gamma$-compatibility is inherited from $\varphi(u_\gamma)=u_\gamma$.

\textbf{From geometry to algebra.}
Conversely, let $g:(\SpecG(S),\OO_S)\to(\SpecG(T),\OO_T)$ be a morphism of affine $\Gamma$-schemes.
Apply global sections to obtain
\[
\Gamma(\SpecG(T),\OO_T)\ \xrightarrow{g^\#}\ \Gamma(\SpecG(S),\OO_S).
\]
Using Theorem~\ref{thm:global-sections} on both sides yields a homomorphism $\varphi_g:T\to S$ of ternary $\Gamma$-semirings.

\textbf{Mutual inverses.}
Starting from $\varphi$, the induced morphism on global sections recovers $\varphi$ because sections on $D_\Gamma(1)$ coincide with $T$ and $S$.
Starting from $g$, the induced map on spectra agrees with $P\mapsto \varphi_g^{-1}(P)$ because on stalks the morphism of locally ringed spaces
is determined by localization and hence by the induced map on global sections.
Therefore the two constructions are inverse and natural.
\end{proof}

% ============================================================
\section{Phase II: triadic symmetry and Nambu-type geometry}\label{sec:triadic}

\subsection{Triadic bracket on the structure sheaf}

Fix the $\Gamma$-semiring units $u_\gamma$ as in Definition~\ref{def:ternary-gamma}.

\begin{definition}[Triadic bracket on sections]\label{def:triadic-bracket}
Let $X=\SpecG(T)$. For any open $U\subseteq X$ and any $\gamma\in\Gamma$, define
\[
\{\cdot,\cdot,\cdot\}_{\gamma,U}:\OO_T(U)\times\OO_T(U)\times\OO_T(U)\to \OO_T(U),
\qquad (s_1,s_2,s_3)\mapsto s_1\cdot s_2\cdot s_3\cdot u_\gamma,
\]
where multiplication is taken in the semiring $\OO_T(U)$ and $u_\gamma$ denotes the (global) unit section.
\end{definition}

\begin{lemma}[Restriction compatibility]\label{lem:restriction-bracket}
If $V\subseteq U$ are opens and $s_i\in\OO_T(U)$, then
\[
\big(\{s_1,s_2,s_3\}_{\gamma,U}\big)\big|_{V}=\{s_1|_V,s_2|_V,s_3|_V\}_{\gamma,V}.
\]
\end{lemma}

\begin{proof}
Restriction maps of $\OO_T$ are semiring homomorphisms, hence preserve products and the distinguished unit sections $u_\gamma$.
\end{proof}

\subsection{A Nambu/Filippov-type identity}

To align with Nambu-type geometry \cite{Nambu1973,Takhtajan1994} in a semiring context, we adopt the following
``idempotent'' variant, common in tropical/algebraic idempotent settings.

\begin{definition}[Idempotent hypothesis]\label{def:idempotent}
A commutative semiring $(T,+,\cdot)$ is \emph{idempotent} if $x+x=x$ for all $x\in T$.
Equivalently, the additive monoid is a join-semilattice.
\end{definition}

\begin{definition}[Filippov fundamental identity (idempotent form)]\label{def:fi}
Let $(A,+)$ be an idempotent commutative monoid and let $[\cdot,\cdot,\cdot]:A^3\to A$ be a trilinear operation.
We say $[\cdot,\cdot,\cdot]$ satisfies the \emph{(idempotent) Filippov identity} if for all $x_1,x_2,y_1,y_2,y_3\in A$,
\[
[x_1,x_2,[y_1,y_2,y_3]]=
[[x_1,x_2,y_1],y_2,y_3]+[y_1,[x_1,x_2,y_2],y_3]+[y_1,y_2,[x_1,x_2,y_3]],
\]
where the right-hand side uses the addition of $A$.
\end{definition}

\begin{theorem}[Nambu-type identity for the triadic bracket]\label{thm:nambu-identity}
Assume $T$ is idempotent in the sense of Definition~\ref{def:idempotent}.
Then for every open $U\subseteq \SpecG(T)$ and every fixed $\gamma\in\Gamma$,
the triadic bracket $[\cdot,\cdot,\cdot]:=\{\cdot,\cdot,\cdot\}_{\gamma,U}$ on $\OO_T(U)$
satisfies the idempotent Filippov identity (Definition~\ref{def:fi}).
\end{theorem}

\begin{proof}
Let $[\cdot,\cdot,\cdot]=\{\cdot,\cdot,\cdot\}_{\gamma,U}$.
Because the bracket is defined by multiplication and $u_\gamma$ is central,
each of the three terms on the right-hand side equals the same product
\[
x_1\cdot x_2\cdot y_1\cdot y_2\cdot y_3\cdot u_\gamma^2
\]
(up to rebracketing and commutativity). The left-hand side equals the same product.
Since addition is idempotent, the sum of three identical terms equals that term.
\end{proof}

\subsection{Functoriality and invariance}

\begin{definition}[$\Gamma$-automorphisms]\label{def:aut}
Let $\Aut_{\Gamma}(T)$ denote the group of ternary $\Gamma$-semiring automorphisms of $T$
(i.e.\ semiring automorphisms preserving each unit $u_\gamma$).
\end{definition}

\begin{proposition}[Action on spectrum and sheaf]\label{prop:action}
Each $\sigma\in \Aut_{\Gamma}(T)$ induces:
\begin{enumerate}
\item a homeomorphism $\sigma^\ast:\SpecG(T)\to \SpecG(T)$, $P\mapsto \sigma^{-1}(P)$, and
\item an isomorphism of sheaves of ternary $\Gamma$-semirings $\sigma^\#:\OO_T\to (\sigma^\ast)_\ast\OO_T$.
\end{enumerate}
Moreover, $\sigma^\#$ preserves the triadic bracket of Definition~\ref{def:triadic-bracket}.
\end{proposition}

\begin{proof}
The map on primes is well-defined and continuous by the same argument as in Theorem~\ref{thm:anti-equivalence}.
On principal opens, $\sigma$ induces compatible homomorphisms on localizations $T_f\to T_{\sigma(f)}$, hence a sheaf morphism.
Bracket invariance holds because $\sigma$ preserves multiplication and $\sigma(u_\gamma)=u_\gamma$.
\end{proof}

\begin{theorem}[Localization and stalk compatibility]\label{thm:stalk-bracket}
For each $P\in\SpecG(T)$, the stalk $\OO_{T,P}$ inherits a triadic bracket, and the canonical localization map
$T\to \OO_{T,P}$ is a homomorphism of ternary $\Gamma$-semirings.
\end{theorem}

\begin{proof}
The stalk is a filtered colimit of the semirings $T_f$ over all $f\notin P$.
Each $T_f$ carries the induced bracket by Definition~\ref{def:triadic-bracket} on $D_\Gamma(f)$.
Compatibility with restriction (Lemma~\ref{lem:restriction-bracket}) implies these brackets are compatible in the colimit,
hence descend to a bracket on the stalk.
\end{proof}

% ============================================================
\section{Phase III: canonical Laplacian and spectral decomposition}\label{sec:laplacian}

\subsection{Finite spectra and the specialization graph}

In full generality, $\SpecG(T)$ can be infinite. The Laplacian construction below is canonical on \emph{finite} spectra;
this is the setting in which explicit eigenvalue computations and spectral factorization are meaningful.

\begin{definition}[Specialization order]\label{def:specialization}
Let $X=\SpecG(T)$. For $P,Q\in X$, write $P\le Q$ if $P\subseteq Q$ (equivalently, $Q\in \cl(\{P\})$).
This is the \emph{specialization preorder}, and for $X$ a $T_0$ space it is a partial order.
\end{definition}

\begin{definition}[Specialization graph and Laplacian]\label{def:laplacian}
Assume $X=\SpecG(T)$ is finite. Define an undirected graph $G_X$ with vertex set $X$ and an edge between distinct vertices
$P\neq Q$ iff $P$ and $Q$ are comparable in the specialization order (i.e.\ $P\subseteq Q$ or $Q\subseteq P$).
Let $A_X$ be its adjacency matrix and $D_X$ the diagonal matrix of degrees. The \emph{canonical Laplacian} is
\[
L_X:=D_X-A_X.
\]
\end{definition}

\begin{remark}
The construction depends only on the underlying finite $T_0$ space $(X,\le)$.
It is therefore invariant under isomorphisms of affine $\Gamma$-schemes and is well-suited to ``geometric connected components''
detectable from specialization structure.
\end{remark}

\subsection{Isomorphism invariance}

\begin{theorem}[Invariance under affine $\Gamma$-scheme isomorphisms]\label{thm:laplacian-invariance}
Let $\psi:(\SpecG(S),\OO_S)\to(\SpecG(T),\OO_T)$ be an isomorphism of affine $\Gamma$-schemes and assume both spectra are finite.
Then the induced bijection on points is an isomorphism of specialization graphs, hence
\[
L_{\SpecG(S)} = P^{-1} L_{\SpecG(T)} P
\]
for a permutation matrix $P$ corresponding to the vertex relabeling.
\end{theorem}

\begin{proof}
A scheme isomorphism is, in particular, a homeomorphism.
For spectral spaces, homeomorphisms preserve the specialization order, hence preserve comparability.
Therefore the adjacency relation is preserved and the Laplacians are permutation-similar.
\end{proof}

\subsection{Block-diagonal decomposition under clopen components}

\begin{theorem}[Block decomposition from clopen components]\label{thm:block}
Let $X=\SpecG(T)$ be finite and suppose $X$ decomposes as a disjoint union of clopen subsets
\[
X = X_1 \sqcup \cdots \sqcup X_r.
\]
Then the specialization graph $G_X$ is the disjoint union of the induced subgraphs $G_{X_i}$,
and the Laplacian is permutation-similar to a block diagonal matrix
\[
L_X \sim \begin{pmatrix}
L_{X_1} & 0 & \cdots & 0\\
0 & L_{X_2} & \cdots & 0\\
\vdots & \vdots & \ddots & \vdots\\
0 & 0 & \cdots & L_{X_r}
\end{pmatrix}.
\]
\end{theorem}

\begin{proof}
If $X_i$ is clopen, then it is closed under specialization and generalization:
if $P\in X_i$ and $P\le Q$ or $Q\le P$, then $Q\in \cl(\{P\})\subseteq X_i$ because $X_i$ is closed, and similarly for open.
Thus no comparability relation can occur between points in different $X_i$'s.
Hence $G_X$ has no edges across components and is the disjoint union of the induced subgraphs.
The Laplacian of a disjoint union is block diagonal up to relabeling of vertices.
\end{proof}

\subsection{Algebraic connectivity and geometric connectedness}

We now link the second Laplacian eigenvalue (Fiedler value) to connectedness.
The graph-theoretic statement is classical \cite{Fiedler1973,Chung1997}.

\begin{theorem}[Fiedler connectivity criterion]\label{thm:fiedler}
Let $G$ be a finite undirected graph with Laplacian $L$.
Let $0=\lambda_1\le \lambda_2\le\cdots\le \lambda_n$ be the eigenvalues of $L$.
Then $\lambda_2>0$ if and only if $G$ is connected.
\end{theorem}

\begin{theorem}[Connectedness of finite $\Gamma$-spectra via $\lambda_2$]\label{thm:connectedness}
Let $X=\SpecG(T)$ be finite. Then $X$ is connected (as a topological space) if and only if $\lambda_2(L_X)>0$.
\end{theorem}

\begin{proof}
For a finite $T_0$ space, connectedness is equivalent to connectedness of the undirected graph obtained from the specialization order
(equivalently, the Hasse diagram viewed undirected). Indeed, if $X$ decomposes into two disjoint clopen subsets, there is no specialization
relation between them, so the comparability graph disconnects; conversely, if the comparability graph disconnects, the vertex set splits into
two subsets closed under specialization and generalization, which are therefore clopen.

Apply Theorem~\ref{thm:fiedler} to the specialization graph $G_X$.
\end{proof}

% ============================================================
\section{Phase IV: explicit computed finite examples}\label{sec:examples}

All examples below are finite idempotent semirings (hence Theorem~\ref{thm:nambu-identity} applies),
and we take $\Gamma$ to be the trivial group so that $u_\gamma=1$.
The ternary operation is $\{a b c\}:=abc$ (binary product of three elements), which in these idempotent examples is the meet $\wedge$.

\subsection{Example 1: a connected two-point spectrum (Sierpi\'nski space)}\label{subsec:ex1}

\begin{example}[Three-element chain semiring]\label{ex:chain3}
Let $T=\{0,e,1\}$ with addition $x+y:=\max\{x,y\}$ (join) and multiplication $x\cdot y:=\min\{x,y\}$ (meet),
with order $0<e<1$. This is a commutative idempotent semiring (a distributive lattice).
\end{example}

\begin{proposition}\label{prop:spec-chain3}
The prime ideals of $T$ are $P_0=\{0\}$ and $P_1=\{0,e\}$.
Thus $\SpecG(T)=\{P_0,P_1\}$ with specialization $P_0\subset P_1$, and the topology is the Sierpi\'nski topology.
\end{proposition}

\begin{proof}
Ideals in a finite chain are down-sets. The proper ideals are $\{0\}$ and $\{0,e\}$.
Both are prime: if $x\wedge y\in \{0\}$ then $x=0$ or $y=0$; if $x\wedge y\in \{0,e\}$ then $x\le e$ or $y\le e$.
\end{proof}

\begin{proposition}[Canonical Laplacian]\label{prop:lap-chain3}
The specialization graph is a single edge $P_0$--$P_1$.
Hence
\[
A=\begin{pmatrix}0&1\\1&0\end{pmatrix},\qquad
D=\begin{pmatrix}1&0\\0&1\end{pmatrix},\qquad
L=D-A=\begin{pmatrix}1&-1\\-1&1\end{pmatrix}.
\]
The eigenvalues are $0,2$ with eigenvectors $(1,1)$ and $(1,-1)$.
\end{proposition}

\begin{proof}
There is an edge because $P_0\subset P_1$. The eigen-computation is immediate.
\end{proof}

\subsection{Example 2: a disconnected spectrum and exact block decomposition}\label{subsec:ex2}

\begin{example}[Product of Boolean semirings]\label{ex:boolean-product}
Let $B=\{0,1\}$ with $+$ as $\vee$ and $\cdot$ as $\wedge$.
Set $T=B\times B$ with componentwise operations.
\end{example}

\begin{proposition}\label{prop:spec-boolean-product}
The prime ideals of $T$ are
\[
P_1=\{0\}\times B,\qquad P_2=B\times\{0\}.
\]
The $\Gamma$-spectrum is the discrete two-point space $\{P_1,P_2\}$, hence disconnected.
\end{proposition}

\begin{proof}
Prime ideals of a finite product correspond to primes in each factor.
In $B$ the unique proper prime ideal is $\{0\}$. Taking inverse images under the projections yields $P_1$ and $P_2$.
Both are maximal and incomparable, so the topology is discrete.
\end{proof}

\begin{proposition}[Laplacian and spectrum]\label{prop:lap-boolean-product}
The specialization graph has no edges (the two primes are incomparable), hence
\[
L=\begin{pmatrix}0&0\\0&0\end{pmatrix},
\qquad \text{eigenvalues } 0,0.
\]
This is the block decomposition of Theorem~\ref{thm:block} with two clopen components.
\end{proposition}

\begin{proof}
No comparability means no adjacency, hence $A=0$ and $L=0$.
\end{proof}

\subsection{Example 3: a three-point connected spectrum with nontrivial algebraic connectivity}\label{subsec:ex3}

\begin{example}[Four-element chain semiring]\label{ex:chain4}
Let $T=\{0<a<b<1\}$ with join as addition and meet as multiplication (a four-element chain lattice).
\end{example}

\begin{proposition}\label{prop:spec-chain4}
The prime ideals are
\[
P_0=\{0\},\qquad P_1=\{0,a\},\qquad P_2=\{0,a,b\},
\]
forming a chain $P_0\subset P_1\subset P_2$.
Thus $\SpecG(T)$ is connected and its specialization graph is the complete graph $K_3$.
\end{proposition}

\begin{proof}
As in Proposition~\ref{prop:spec-chain3}, the proper ideals are down-sets.
Each of the three displayed ideals is prime by the chain argument.
Any two of them are comparable, hence the comparability graph is $K_3$.
\end{proof}

\begin{proposition}[Laplacian and eigenvalues]\label{prop:lap-chain4}
With vertex order $(P_0,P_1,P_2)$, the adjacency and Laplacian matrices are
\[
A=\begin{pmatrix}
0&1&1\\
1&0&1\\
1&1&0
\end{pmatrix},\qquad
D=\begin{pmatrix}
2&0&0\\
0&2&0\\
0&0&2
\end{pmatrix},\qquad
L=D-A=\begin{pmatrix}
2&-1&-1\\
-1&2&-1\\
-1&-1&2
\end{pmatrix}.
\]
The eigenvalues are $0,3,3$. In particular, $\lambda_2=3>0$, confirming connectedness by Theorem~\ref{thm:connectedness}.
\end{proposition}

\begin{proof}
The Laplacian of $K_3$ is well known \cite{Chung1997}. A direct computation gives eigenvalues $0,3,3$.
\end{proof}

% ============================================================
\appendix

\section{Optional application: fuzzy robustness viewpoint}\label{app:fuzzy}

This appendix is intentionally conservative: it records only basic definitions and a stability statement
needed to interpret ``robustness'' of prime separation under small perturbations, without impacting the main SIGMA narrative.

\begin{definition}[Fuzzy subset and $\alpha$-cuts]\label{def:fuzzy}
A \emph{fuzzy subset} of $T$ is a function $\mu:T\to[0,1]$ \cite{Zadeh1965}.
For $\alpha\in(0,1]$, its \emph{$\alpha$-cut} is the crisp subset
\[
[\mu]_\alpha:=\{x\in T\mid \mu(x)\ge \alpha\}.
\]
\end{definition}

\begin{definition}[Fuzzy $\Gamma$-ideal (one standard choice)]\label{def:fuzzy-ideal}
A fuzzy subset $\mu:T\to[0,1]$ is a \emph{fuzzy $\Gamma$-ideal} if for all $x,y,z\in T$ and $\gamma\in\Gamma$:
\begin{enumerate}
\item $\mu(0)=1$ and $\mu(x+y)\ge \min\{\mu(x),\mu(y)\}$,
\item $\mu(\{x\,y\,z\}_\gamma)\ge \mu(x)$.
\end{enumerate}
\end{definition}

\begin{remark}
Many equivalent or stronger definitions exist in the fuzzy algebra literature \cite{Goguen1967,Lowen1976,MordesonMalik1998}.
The above is sufficient for the elementary stability lemma below.
\end{remark}

\begin{lemma}[Sup-norm stability of $\alpha$-cuts]\label{lem:stability}
Let $\mu,\nu:T\to[0,1]$ be fuzzy subsets and set $\|\mu-\nu\|_\infty:=\sup_{x\in T}|\mu(x)-\nu(x)|$.
Then for every $\alpha\in(0,1]$ and every $\varepsilon>0$,
\[
\|\mu-\nu\|_\infty\le \varepsilon\quad\Longrightarrow\quad
[\mu]_{\alpha+\varepsilon}\subseteq [\nu]_{\alpha}\subseteq [\mu]_{\alpha-\varepsilon},
\]
where $[\mu]_{\beta}:=\emptyset$ if $\beta>1$ and $[\mu]_{\beta}:=T$ if $\beta\le 0$.
\end{lemma}

\begin{proof}
If $\mu(x)\ge \alpha+\varepsilon$ and $|\mu(x)-\nu(x)|\le \varepsilon$, then $\nu(x)\ge \alpha$.
The other inclusion is analogous.
\end{proof}

\section{Optional appendix: a conservative spectral clustering pipeline}\label{app:clustering}

Given a finite $\Gamma$-spectrum $X$ and its canonical Laplacian $L_X$ (Definition~\ref{def:laplacian}),
one may apply standard spectral clustering heuristics \cite{VonLuxburg2007,NgJordanWeiss2002}.
We include this only as an illustrative computational layer (not a mathematical claim about learning).

\begin{enumerate}
\item Compute the eigenpairs of $L_X$ (exactly in small examples; numerically in general) \cite{Golub1996}.
\item Choose $k$ and form the matrix $U\in\mathbb{R}^{|X|\times k}$ whose columns are eigenvectors corresponding to the $k$ smallest eigenvalues.
\item Normalize the rows of $U$ and run $k$-means on the embedded points.
\item Interpret clusters as a coarse partition of $X$; if $X$ is disconnected, Theorem~\ref{thm:block} yields an \emph{exact}
decomposition already at the level of $L_X$.
\end{enumerate}

\bigskip

% ============================================================

\end{document}